\documentclass[12pt]{amsart}
\date{June~8, 2010}
\usepackage{setspace}

\calclayout

\usepackage{color}
 \usepackage{amssymb}
 \usepackage[matrix,arrow,curve]{xy}

\newtheorem{thm}{anything}[section]

\newtheorem{Example}[thm]{Example}
\newtheorem{Examples}[thm]{Examples}

\newtheorem*{thma}{Theorem A}
\newtheorem*{thma-top}{Theorem B}

\newtheorem{Lemma}[thm]{Lemma}

\newtheorem{Corollary}[thm]{Corollary}
\newtheorem{Proposition}[thm]{Proposition}

\newtheorem{Remark}[thm]{Remark}

\theoremstyle{definition}
\newtheorem{definition}[thm]{Definition}
  \newtheorem{example}[thm]{Example}
  \newtheorem{ccote}[thm]{\hskip -0.6mm}
  \newtheorem{remark}[thm]{Remark}
   
     \newtheorem*{remarkB}{Remark}
    \newtheorem*{acknowledgement}{Acknowledgements}
  
\newcommand
{\eqncount}{\setcounter{equation}{\value{thm}}%
\addtocounter{thm}{1}}

\newcommand{\fl}[1]{\buildrel{#1}\over{\longrightarrow}}

\newcommand{\cqfd}{\unskip\kern 6pt\penalty 500
\raise -2pt\hbox{\vrule\vbox to10pt{\hrule width
 4pt\vfill\hrule}\vrule}\smallskip}

\newcommand{\proref}[1]{Proposition~\ref{#1}}
\newcommand{\remref}[1]{Remark~\ref{#1}}
\newcommand{\lemref}[1]{Lemma~\ref{#1}}
\newcommand{\corref}[1]{Corollary~\ref{#1}}

\newcommand{\exref}[1]{Example~\ref{#1}}

\newcommand{\secref}[1]{Section~\ref{#1}}
\DeclareMathOperator{\sign}{sign}
\DeclareMathOperator{\Sharp}{\#}
\newcommand{\la}{\langle}
\newcommand{\ra}{\rangle}
\DeclareMathOperator{\Image}{im}

\newcommand{\identity}{1}

\newcommand{\bbr}{{\mathbb{R}}}

\newcommand{\bbc}{{\mathbb{C}}}

\newcommand{\bbz}{{\mathbb{Z}}}
\newcommand{\bbq}{{\mathbb{Q}}}
\newcommand{\bbn}{{\mathbb{N}}}

\newcommand{\calc}{{\mathcal C}}

\newcommand{\caln}{{\mathcal N}}

\newcommand{\CP}{\bbc P}
\newcommand{\RP}{\bbr P}

\newcommand{\pcirc}{\circ}

\newcommand{\mun}{{-1}}

\newcommand{\onto}{\to\kern-7.5pt\to}
\newcommand{\donto}{\downarrow\kern -7.92pt\raisebox{-0.6ex}{$\downarrow$}}

\newcommand{\st}[1]{\vskip 2mm\noindent \textit{Step #1.} \hskip 1mm }

\newcommand{\scr}{\scriptscriptstyle}

\newcommand{\intb}{\,\raisebox{7.5pt}{$\scr\circ$}\kern -5pt}
\newcommand{\mancqfd}{\hfill \ensuremath{\Box}}

\renewcommand{\:}{\colon}

\newcommand{\hs}{$\bbz_2$-homology sphere}
\newcommand{\csp}{conjugation space}

\newcommand{\hfra}{$H^\ddstar$-frame}
\newcommand{\ef}{equivariantly formal}

\newcommand{\sq}{{\rm Sq}}
\newcommand{\ztk}{$\bbz_2$-knot}
\newcommand{\rp}{\bbr P^{2}}
\newcommand{\cp}{\bbc P^{2}}

\newcommand{\zloc}{\bbz_{(2)}}

\newcommand{\dstar}{\lower .8 pt \hbox{{$\displaystyle *$}}}
\newcommand{\ddstar}{{\lower 1.7 pt \hbox{{$\displaystyle *$}}}}

\begin{document}

\title{Conjugation spaces and  $4$-manifolds}
\author{Ian HAMBLETON}
\address{Department of Mathematics \& Statistics
\newline\indent
McMaster University
\newline\indent 
Hamilton, Ontario L8S 4K1, Canada}

\email{hambleton@mcmaster.ca }
\author{Jean-Claude HAUSMANN}
\address{Section de Math\'ematiques
\newline\indent
Universit\'e de Gen\`eve,
 B.P.~240
\newline\indent
CH-1211 Gen\`eve 24,
Switzerland}
\email{Jean-Claude.Hausmann@unige.ch}

\thanks{Research partially supported by NSERC Discovery Grant A4000. 
The first author would like  to thank the Max Planck Institut f\"ur Mathematik in Bonn.  Both authors thank the Swiss National Funds for Scientific Research for its support.}

\begin{abstract}\noindent
We show that  $4$-dimensional conjugation manifolds are all obtained from branched $2$-fold coverings of  knotted surfaces in $\bbz_2$-homology $4$-spheres. 
\end{abstract}

 \maketitle

\section{Introduction}
Flag manifolds $X$ with complex conjugation,
the Chevalley involution on coadjoint orbits of 
compact Lie groups,  and involutions  on toric manifolds or polygon spaces all share a remarkable  property. Let $G$ denote the group of order 2.
There is a ring isomorphism 
$$
\kappa\colon H^{2\dstar}(X;\bbz_2)\cong H^{\ddstar}(X^{G};\bbz_2)
$$
dividing the degrees in half, where $X^G$ denotes the fixed set under the involution.  The structure underlying this
 property was discovered by Hausmann, Holm and Puppe \cite{hausmann-holm-puppe1}, 
and studied further in \cite{franz-puppe1, hausmann-holm1, olbermann1,  olbermann2, olbermann3}. 
A $G$-space with this structure is a \emph{conjugation space} (see \secref{S.preli} for the precise definition). 
 
 \smallskip
In this paper, we study the interaction between conjugation structures and the topology of smooth $4$-manifolds.
A \textit{conjugation $4$-manifold} is a smooth closed $G$-manifold $X$ of dimension $4$ which is a conjugation
space. The fixed point set $X^{G}$ is  a closed connected  surface embedded in $X$.  In addition, $X$ has no odd-degree cohomology (${\rm mod\,}2)$, and hence
a conjugation manifold  is orientable.

Let $X$ be an oriented conjugation $4$-manifold. 
The quotient space $X/G$ inherits  a canonical smooth structure (see \lemref{smoothquo}),
and thus $X/G$ is an oriented
closed smooth $4$-manifold containing the surface $X^G$ as a smooth submanifold. 

A {\it \ztk\ } is a smooth manifold pair $(M,\Sigma)$, where $M$ is an oriented $4$-dimensional \hs\ and $\Sigma$ is a closed connected surface embedded in $M$.

\begin{thma}
The correspondence $X\mapsto (X/G,X^G)$ defines a bijection between
\renewcommand{\labelenumi}{(\alph{enumi})}
\begin{enumerate}
\item the orientation-preserving $G$-diffeomorphism classes of oriented
connected conjugation $4$-manifolds,
and
\item
the smooth equivalence classes of \ztk s.
\end{enumerate}
\end{thma}
Two \ztk s $(M,\Sigma)$ and $(M',\Sigma')$ are {\it smoothly equivalent}
if there is an orientation-preserving diffeomorphism $h\:M\to M'$
such that $h(\Sigma)=\Sigma'$.

The inverse of the bijection in Theorem A is provided by taking a branched $2$-fold covering  of $M$ over the knot $\Sigma$. We therefore
 need to understand the relation between smooth manifold structures on the total spaces and quotients of branched $2$-fold coverings, with codimension two branch locus (see the Appendix \S \ref{S.bran}).
 Other versions of Theorem~A are given in \secref{S.top}, for instance for topological manifolds (Theorem B),  or for non-oriented manifolds.

\smallskip
Under the bijection of Theorem~A, any knot $S^2\hookrightarrow S^4$ corresponds to a conjugation
$4$-manifold $X$ with $X^G\approx S^2$. 
For the trivial knot $S^2\subset S^4$, $X$ is the sphere $S^4$ on which $G$ acts by a  
linear involution with $2$ negative eigenvalues (see 
\exref{E.TrivKnot}). In general, $X$ is not simply connected. 
On the other hand, Gordon \cite{gordon1}, \cite{gordon2} and Sumners \cite{sumners1} found
infinitely many topologically distinct knots in $S^4$
which are the fixed point set of smooth involutions
(contrasting with the Smith conjecture in dimension $3$), and earlier examples on homotopy
$4$-spheres were found by Giffen \cite{giffen1}. Our work adds a new perspective: the examples of 
Gordon and Sumners produce infinitely many
topologically inequivalent smooth conjugations on $S^4$ (see \secref{S.pi1}).

The classical examples of conjugation $4$-manifolds come from the complex conjugations on 
$S^2\times S^2$, with fixed point set  $S^1\times S^1$,  and on $\CP^2$ (or $\overline{\CP}^2$),  with fixed point set 
$\RP^2$. By taking connected sums along the fixed sets, one can thus realize any closed surface as the fixed point set of a conjugation $4$-manifold. These classical examples all have quotient a smooth manifold diffeomorphic to $S^4$ (see Arnold \cite{arnold1}, Kuiper \cite{kuiper1}, Letizia \cite{letizia1}, and  Massey \cite{massey3}). 
 For the reader's convenience, in \proref{classic1} we include a proof using classification results for group actions by Bredon \cite{bredon1} and Orlik-Raymond \cite{orlik-raymond1} (but not the deep results of Freedman \cite{freedman1} or Cerf \cite{cerf1}). 
 
If $X$ is any simply-connected conjugation
$4$-manifold, we prove in \proref{1conncon-bis} that $X/G$ is at least homeomorphic to $S^4$. 
In addition, we show in \proref{P.spinStan} that $X$  is homeomorphic to a connected sum of copies of $S^2\times S^2$, $\CP^2$, and $\overline{\CP}^2$
(but not necessarily equivariantly).  For example, the K3 surface  does not admit a conjugation structure.
\begin{remarkB} A  conjugation $4$-manifold $X$ is \emph{equivariantly minimal}  among $G$-actions on $4$-manifolds with a given surface as the fixed set, since $X$ can not be decomposed as a non-trivial equivariant connected sum in the free part of the $G$-action (see Proposition \ref{P.munH2}).
\end{remarkB}
For the remainder of the paper, the cohomology 
$H^\ddstar(-)=H^\ddstar(-;\bbz_2)$ is taken with
coefficients in the field  $\bbz_2$, unless otherwise mentioned. The letter $G$ stands
for the group of order $2$, with $G=\{\identity,\tau\}$, and a $G$-space is a space together
with an involution $\tau$.

\begin{acknowledgement}
The authors would like to thank Allan Edmonds,   Ron Fintushel, Cameron Gordon,  Slava Kharlamov, Volker Puppe, Ron Stern and Claude Weber for helpful conversations and correspondence, and the referee for
valuable suggestions. 
\end{acknowledgement}

\section{Conjugation spaces and manifolds}\label{S.preli}
For a $G$-space $X$, the 
equivariant cohomology $H^\ddstar_G(X)$ is defined as the (singular) cohomology of
the Borel construction:
$$
H^\ddstar_G(X)=H^\ddstar(X\times_G EG) \, .
$$
Hence, $H^\ddstar_G(X)$ is a $H^\ddstar(BG)$-algebra via the projection $X\times_G EG \to BG$.
Since $G$ is the group of order two, $BG=\RP^\infty$ and
$H^\ddstar(BG)=\bbz_2[u]$, with $u$ in degree $1$. Thus
$H^\ddstar_G(X)$ is a $\bbz_2[u]$-algebra.
Let $\rho\: H^{\ddstar}_G(X)\to H^{\ddstar}(X)$ and
$r\: H^{\ddstar}_G(X)\to H^{\ddstar}_G(X^G)$ be the restriction homomorphisms.
As $G$ acts trivially on $X^G$, one has $(X^G)_G=BG\times X^G$,
whence a canonical ring isomorphism $H^\ddstar_G(X^G)=H^\ddstar(X^G)[u]$.

\begin{ccote} \textbf{Conjugation spaces} (\cite{hausmann-holm-puppe1}). 
A {\it cohomology frame} or {\it \hfra}\
for a $G$-space $X$ is a pair $(\kappa,\sigma)$, where
\renewcommand{\labelenumi}{(\alph{enumi})}
\begin{enumerate}
\item $\kappa\:  H^{2m}(X)\to H^{m}(X^G)$, $m \geq 0$, is
an additive isomorphism dividing the degrees in half,  and
\item $\sigma\:  H^{2m}(X)\to H^{2m}_G(X)$, $m\geq 0$, is
an additive section of $\rho$.
\end{enumerate}
Moreover, $\kappa$ and $\sigma$ must satisfy the {\it conjugation equation}
\eqncount
\begin{equation}\label{defalceq}
r\pcirc\sigma(a) = \kappa(a)u^m + \ell_m(u)
\end{equation}
for all $a\in H^{2m}(X)$ and all $m\in\bbn$, where $\ell_m(u)$
denotes any polynomial in the variable $u$ of degree less than $m$.
An involution admitting a \hfra\ is called a {\it conjugation}.
A $G$-space $X$ such that $H^{odd}(X)=0$ and admitting an \hfra\ is called a 
{\it \csp}.

Here below are some important properties of \csp s.

\begin{enumerate}
\item If $(\kappa,\sigma)$ is \hfra, then $\kappa$ and $\sigma$ are ring homomorphisms \cite[Theorem~3.3]{hausmann-holm-puppe1}.
The ring homomorphism $\kappa$ also commutes with the Steenrod squares:
$\kappa\pcirc\sq^{2i}=\sq^i\pcirc\kappa$, \cite[Theorem~1.3]{franz-puppe1}.
\item \hfra s are natural for $\tau$-equivariant maps \cite[Prop.~3.11]{hausmann-holm-puppe1}. In particular, if an involution
admits an \hfra, it is unique \cite[Cor.~3.12]{hausmann-holm-puppe1}.
\item For a conjugate-equivariant complex vector bundle $\eta$ (``real bundle''
in the sense of Atiyah) over a \csp\ $X$, the isomorphism $\kappa$
sends the total Chern class of $\eta$ onto the total Stiefel-Whitney class of
its fixed bundle.
\end{enumerate}
\end{ccote}

\begin{ccote}\label{eqform} \textbf{Equivariantly formal spaces}.
A $G$-space $X$ is {\it \ef\ (over $\bbz_2$)} if
the restriction homomorphism $\rho\:H^\ddstar_G(X)\to H^\ddstar(X)$ is surjective.
For instance, a \csp\ is \ef.
The following result is proved in \cite[Prop.~1.3.14]{allday-puppe1}.

\begin{Proposition}\label{P.ef}
Let $X$ be finite dimensional $G$-CW-complex with $\sum b_i(X)$ finite, where $b_i(-)=\dim H^i(-)$. The following statements are
equivalent.
\begin{enumerate}
\item $X$ is \ef. 
\item $\sum b_i(X) = \sum b_i(X^G)$.
\item The restriction homomorphism $H^\ddstar_G(X)\to H^\ddstar_G(X^G)$ is injective.
\end{enumerate}
\end{Proposition}
\end{ccote}

\begin{remark} \label{R.finGsp}
A smooth $G$-manifold has the equivariant homotopy type of a finite $G$-CW complex \cite{illman1}.
The assumptions of Proposition \ref{P.ef} are also satisfied for $X$ a closed topological manifold with a \textit{locally smooth} $G$-action. See Kwasik \cite{kwasik1} for this statement and further references. 
\end{remark}

Here is a consequence of \proref{P.ef}.

\begin{Proposition}\label{L.conjSn}
Let $X$ be a finite $G$-CW-complex. Suppose that $H^\ddstar(X)\approx H^\ddstar(S^{2n})$ and 
$H^\ddstar(X^G)\approx H^\ddstar(S^{n})$. Then $X$ is a conjugation space.
\end{Proposition}

\begin{proof}
By \proref{P.ef}, $X$ is \ef\ and the restriction homomorphism $r\: H^\ddstar_G(X)\to H^\ddstar_G(X^G)$ is injective.
Let $a\in H^{2n}(X)$ and $b\in H^n(X^G)$ be the generators. Let 
$\sigma\:H^{2n}(X)\to H^{2n}_G(X)$ be a section of $\rho\:H^{2n}_G(X)\to H^{2n}(X)$. 
Since $H^{2n}_G(X^G)$ is generated by $bu^n$ and $u^{2n}$,
one has $r\pcirc\sigma=\lambda_1bu^n + \lambda_2u^{2n}$. Setting $\sigma'(a)=\sigma(a)+\lambda_2u^{2n}$
produces a new section $\sigma'$ with $r\pcirc\sigma'(a)=b u^n$. Hence, $X$ is a conjugation space. 
\end{proof}

\begin{ccote} \textbf{Conjugation manifolds}. 
A \textit{conjugation manifold} is a smooth closed manifold equipped with a smooth involution
which is  a conjugation. As $H^{odd}(X)=0$, $X$ must be orientable and of dimension~$2n$. The fixed point set $X^G$ is a closed smooth manifold of dimension~$n$.  Also, the involution of a conjugation manifold preserves connected components (see \cite[Remark~3.1]{hausmann-holm-puppe1}),
so one can restrict to connected manifolds. 

There are natural questions in all dimensions:
\begin{enumerate}
\renewcommand{\labelenumi}{(\roman{enumi})}
\item\label{question1} \emph{Given a closed connected smooth manifold $M^n$, does there exists a conjugation $2n$-manifold $X$ with $X^G$ diffeomorphic to $M$ ?}
\item \emph{Given a closed connected smooth $2n$-manifold $X$, does $X$ admit a 
smooth conjugation structure ?}
\item \emph{How can one classify conjugation manifolds up to $G$-diffeomorphism ?}
\end{enumerate}
Theorem A provides an answer for question (iii) in dimension $4$. The remainder of this section contains some partial results on questions (i) and (ii).
\begin{remark}\label{rem: cases} 
The circle is the fixed point of a unique conjugation $2$-manifold,
namely $S^2$ with a reflection through the equator. 
The uniqueness comes from the classical result that a continuous involution on $S^2$
is topologically conjugate to a linear one, see Constantin and Kolev \cite[Theorem 4.1]{constantin-kolev1}.
For a smooth conjugation, the uniqueness follows from the smooth Sch\"onflies theorem. 
The work of Olbermann \cite{olbermann1,  olbermann2, olbermann3} addresses  questions (i)-(iii) for $6$-manifolds.  

We have already seen that any closed surface can be the fixed point set of a conjugation manifold. However, the answer to question (i) can be negative without further assumptions on $M^{n}$, for $n>2$. For example, W.~Pitsch  and J.~Scherer observed that the Cayley projective plane is
a closed 16-dimensional manifold \cite[Theorem~7.21, p.~707]{whitehead-gw2}, 
which can  not be the fixed point set of any conjugation space. Indeed, a famous theorem
of Adams and its proof \cite[Theorem~1.1.1]{adams0} shows that
$\bbz_2[x]/(x^3)$ is not the $\bbz_2$-cohomology ring of a space if ${\rm degree\,}(x)>8$.
\end{remark}

If $X$ is a manifold, we denote by $v_i(X)$ and $w_i(X)$ in $H^{i}(X)$ 
its  Wu and Stiefel-Whitney classes. \proref{P.WuSW} below and its corollaries were also noticed by Pitsch and Scherer. 

\begin{Proposition}\label{P.WuSW}
Let $X$ be a smooth conjugation manifold of dimension $2n$, with \hfra\ $(\kappa,\sigma)$.
Then $\kappa(v_{2i}(X))=v_i(X^G)$ and $\kappa(w_{2i}(X))=w_i(X^G)$.
\end{Proposition}
 
\begin{proof}
The Wu class $v_{2i}(X)$ is characterised by the equation
\eqncount
\begin{equation}\label{E.defvi}
 v_{2i}(X) \smile a = \sq^{2i}(a)  \ \hbox{ for all } a\in H^{2n-2i}(X) \, .
\end{equation}
The ring isomorphism $\kappa\:H^{2\dstar}(x)\to H^\ddstar(X^G)$ satisfying 
$\kappa\pcirc\sq^{2i}=\sq^i\pcirc\kappa$, \cite[Theorem~1.3]{franz-puppe1}.
Applying $\kappa$ to~\eqref{E.defvi} thus gives
\eqncount
\begin{equation}\label{E.kappavi}
 \kappa(v_{2i}(X)) \smile \kappa(a) = \sq^{i}(\kappa(a))  \ 
\hbox{ for all } a\in H^{2n-2i}(X) \, .
\end{equation}
As $\kappa$ is bijective, \eqref{E.kappavi} implies that 
$$
 \kappa(v_{2i}(X)) \smile b = \sq^{i}(b)  \ 
\hbox{ for all } b\in H^{n-i}(X^G) \, ,
$$
which implies that $\kappa(v_{2i}(X))=v_i(X^G)$. As $H^{odd}(X)=0$, 
the Wu formula says that
\eqncount
\begin{equation}\label{E.WuFor}
w_{2i}(X) = \sum_{k=1}^i \sq^{2i-2k}\,  v_{2k}(X) \, .
\end{equation}
Applying $\kappa$ to \eqref{E.WuFor} and using that $\kappa(v_{2i}(X))=v_i(X^G)$,
we get 
$$
\kappa(w_{2i}(X)) = \sum_{k=1}^i \sq^{i-k}\, \kappa(v_{2k}(X)) = 
\sum_{k=1}^i \sq^{i-k} \, v_{k}(X^G) \, .
$$
By the Wu formula for $X^G$, this implies that $\kappa(w_{2i}(X))=w_i(X^G)$.
\end{proof}

The following corollary may be compared with \cite[Theorem~3]{edmonds0}. Note that,
if a conjugation manifold $X$ is spin, it has a unique spin structure
since $H^1(X)=0$. 

\begin{Corollary}\label{C.VuSW}
Let $X$ be conjugation manifold of dimension $2n$. 
Then $X$ is spin if and only if $X^G$ is orientable.
\end{Corollary}

\begin{proof}
As $H^{odd}(X)=0$, $X$ is spin if and only if $w_2(X)=0$. The results thus follows 
from \proref{P.WuSW}.
\end{proof}

Another corollary concerns non-oriented bordism.

\begin{Corollary}\label{C.2VuSW}
Let $X$ be a conjugation manifold. Then $X$ bounds a compact (possibly non-oriented) manifold if
and only if $X^G$ does so.
\end{Corollary}

\begin{proof}
By theorems of Pontrjagin and  of Thom \cite[pp.~52-53]{milnor-stasheff1}, 
a manifold bounds if and only if
all its Stiefel-Whitney numbers vanish. As $H^{odd}(X)=0$, \proref{P.WuSW}
implies that the collections of the Stiefel-Whitney
numbers for $X$ and $X^G$ are in bijection.
\end{proof}

Since a surface bounds if and only if its Euler characteristic is even, the same statement
holds true, by \corref{C.2VuSW}, for a conjugation $4$-manifold. Actually, 
any orientable $4$-manifold satisfies $w_2^2=w_4$ by Wu's formula, so it bounds if and only if
its Euler characteristic is even. 

The following proposition will be useful.
\begin{Proposition}\label{P.Cma4n}
Let $X$ be a smooth closed connected $G$-manifold of dimension $4n$.
Suppose that $H^k(X)=0$ for $0<k<2n$. Then, the following statements
are equivalent
\begin{enumerate}
\item $X$ is a conjugation manifold.
\item $X^G$ is a $2n$-manifold and $b_n(X^G)\geq b_{2n}(X)$.
\end{enumerate}
\end{Proposition}

Observe that, in general, the existence of an abstract ring isomorphism
from $H^{2*}(X)$ and
$H^*(X^G)$ does not imply that $X$ is a conjugation manifold (see
\cite[Example~1]{franz-puppe1}).

\begin{proof}
Obviously, (1) implies (2). Also, by  Poincar\'e duality,
the condition $H^k(X)=0$ for $0<k<2n$ implies that $H^{odd}(X)=0$.

Suppose that (2) holds true.
Let $X_0$ obtained from $X$ by removing a small open $G$-invariant
$4n$-disk containing
a fixed point. Then $X_0$ is a $G$-subspace with $X_0^G$ equal to  $X^G$ minus
an open $2n$-disk.
Hence,
\begin{equation}\label{P.Cma4n-eq10}
\dim H^*(X_0^G) \geq b_0(X_0^G) + b_n(X_0^G) \geq  1 + b_{2n}(X_0) =  \dim
H^*(X_0) \, .
\end{equation}
On the other hand, $\dim H^*(X_0^G)\leq \dim H^*(X_0)$ by Smith theory
(see \cite[Corollary~1.3.8]{allday-puppe1}).
Therefore, all the inequalities occuring in~\eqref{P.Cma4n-eq10} are
equalities,
implying  $b_0(X^G)=1$, $b_n(X_0^G)=b_{2n}(X_0)$ and $H^k(X_0^G)=0$ for
$0<k<n$.
Also, by \proref{P.ef}, $X_0$ is \ef\ and
$$
H^{2n}_G(X_0^G) = H^n(X_0^G)\, u^n \oplus \bbz_2\, u^{2n} \, .
$$
Choose a section $\sigma\:H^{2n}(X_0)\to H^{2n}_G(X_0)$ of
$\rho\:H^{2n}_G(X_0)\to H^{2n}(X_0)$.
Let $\phi\: H^{2n}(X_0) \to \bbz_2$ be defined by letting $\phi(b)$ denote
the coefficient of $u^{2n}$ in  $(r\circ \sigma)(b)$. Then
by changing $\sigma(b)$
into $\sigma'(b) = \sigma(b)+ \phi(b)\,u^{2n}$,
one may assume that the image $\Image(r\pcirc\sigma)\subseteq H^n(X_0^G)\,
u$.
As $b_{2n}(X_0)=b_n(X_0^G)$,
$r\pcirc\sigma(a)$ is of the form $\kappa(a)u$ for an isomorphism
$\kappa\:H^2(X_0)\to H^1(X_0^G)$. Hence $X_0$ is a conjugation space.
Now, the closure of the small $4n$-disk removed in $X$ is what is called
a \textit{conjugation cell} in \cite[Section~5.1]{hausmann-holm-puppe1}.
By \cite[Prop.~5.1]{hausmann-holm-puppe1},
attaching a conjugation cell (by a $G$-map) to a \csp\ produces a \csp.
Therefore, $X$ is a \csp.
\end{proof}

We now restrict our attention to conjugation $4$-manifolds. Here, the $G$-action preserves the orientation (as can be seen on the tangent space to a fixed point).
In the statement below, $\zloc$ denotes $\bbz$ \textit{localized at $2$}, the smallest subring of $\bbq$
where all odd primes are invertible. 

\begin{Proposition}\label{P.munH2} 
Let $X$ be a smooth $G$-manifold of dimension $4$ with $H_1(X;\zloc)=0$.
Then, $X$ is a conjugation $4$-manifold if and only if $X^G\neq\emptyset$ and $G$ acts on $H^2(X;\zloc)$
as multiplication by $-1$.
\end{Proposition}
 
\begin{Remark}\rm
The condition $H_1(X;\zloc)=0$ is equivalent to $H^1(X;\bbz_2)=0$. The condition on
$H^2(X;\zloc)$ is then equivalent to  $\tau$ acting as multiplication by $-1$
on $H^*(X;\bbz)$ modulo torsion.
\end{Remark}

\begin{proof} If $X$ is simply-connected  with $G$ acting as  $\tau_*=-1$ on $H^2(X;\bbz)$, then $X$ is a conjugation manifold by results of V.~Puppe (see \cite[Theorem~5 and Remark~2]{puppe1}). We note that the same arguments (which are all $2$-local) prove that $X$ is a conjugation $4$-manifold under our weaker assumption.
Again, if $X$ were simply-connected then the other direction would follow from results of A.~Edmonds \cite[2.4]{edmonds2}. We leave the reader to verify our claim that the arguments of \cite[2.1-2.4]{edmonds2} are $2$-local, and again hold under our weaker assumption. It follows that the number of $\zloc$-summands in $H^2(X;\zloc)$ on which $\tau_* =-1$ is equal to the rank of $H^1(\Sigma;\bbz_2)$, which equals the rank of $H^2(X;\bbz_2)$. Hence we have $\tau_*=-1$ on all of $H^2(X;\zloc)$.
\end{proof}

\begin{remark}\label{R.spin}
If $X^G$ is orientable then its integral fundamental class $[X^G]$
represents a $\tau_*$-fixed class in $H_2(X;\bbz)$ modulo odd torsion.
But $\tau_* = -1$ on this quotient, so $X^G$ is null-homologous ${\rm
mod\, } 2$ in $X$. On the other hand, if $X^G$ is non-orientable then $w_2(X)
\neq 0$ by \corref{C.VuSW} and,
by \cite[Cor.~5.2]{edmonds2}, the ${\rm mod\, }2$ homology class of
the fixed set $X^G$ represents the Poincar\'e dual of $w_2(X)$. We
conclude that $X^G$ is null-homologous ${\rm mod\, } 2$
in $X$ if and only if $X^G$ is orientable, or equivalently if and only
if $X$ is spin.
\end{remark}

\begin{Proposition}\label{P.spinStan}
Let $X$ be a simply connected smooth conjugation $4$-manifold.  Then $X$ is (non-equivariantly) homeomorphic to a connected sum of copies of $S^2\times S^2$, if $X$ is spin, or copies of $\CP^2$ and $\overline{\CP}^2$, if $X$ is non-spin.
\end{Proposition}

\begin{proof} In the non-spin case, if $X$ has a (positive) definite intersection form then $X$ is homeomorphic to a connected sum of copies of $\CP^2$ (by Donaldson \cite{donaldson1} and Freedman \cite{freedman1}). If $X$ has an indefinite intersection form then $X$ is homeomorphic to a connected sum of $\CP^2$'s and $\overline{\CP}^2$'s by Freedman's Theorem and the classification of odd unimodular indefinite forms. 

If $X$ is spin, 
we use the equivariant Hirzebruch formula \cite[Formula~(6)]{hirzebruch1}, \cite[Prop.~6.15]{atiyah-III}: 
$$
\sign(X, \tau) = \sign((X^\tau)^2) 
$$
where the right-hand side is given by evaluating the twisted Euler class of the normal bundle of $\Sigma$ in $X$. By \proref{P.munH2}, $G$ acts on $H^2(X;\bbz)$ by multiplication by $-1$.
Therefore, $\sign(X) = -\sign(X,\tau)$. 
As $X$ is spin, the manifold $X^G$ is orientable, by \corref{C.VuSW}.  
By \remref{R.spin}, the integral homology class represented by $X^G$ is zero, 
and hence $\sign(X) = 0$. We deduce that the (even) integral intersection form of $X$  is a sum of  hyperbolic forms and apply 
Freedman's theorem again.
\end{proof}

\begin{remark} The results of \proref{P.munH2} and Remark \ref{R.spin}   hold also for topological conjugation $4$-manifolds, if the involution is assumed to be locally linear (see Section \ref{S.top}). The corresponding result to \proref{P.spinStan} is true in the spin topological locally-linear case: note that the index formula holds in this context \cite[14B]{wall-book} and the Kirby-Siebenmann invariant vanishes because $X$ is spin with $\sign(X)=0$. In the non-spin case, we don't know what happens if $X$ has a definite intersection form.  
\end{remark}
\end{ccote}

\section{The proof of Theorem A}\label{S.proofThmA}

We divide the two directions of the proof into separate lemmas. 

\begin{Lemma}\label{con2knot}
Let $X$ be a connected, oriented, $4$-dimensional conjugation manifold.
Then $(X/G,X^G)$ is a \ztk.
\end{Lemma}

\begin{proof}
As $X$ is a $4$-dimensional conjugation manifold, the
fixed point set $X^G$ is a closed connected surface which we call $\Sigma$.
Let $V$ be a closed $G$-invariant tubular neighbourhood of $\Sigma$ in $X$ and let
$K$ be the complement
of the interior of $V$. Now, $M=X/G$ is a smooth manifold by \lemref{smoothquo}.
As noted in the introduction, $\tau$ preserves the orientation, so
$M$ inherits an orientation 
and the projection map $p\:X\to M$ is smooth, of degree $2$. 
We identify $\Sigma$ with $p(\Sigma)$.
Then $\bar V = p(V)$ is a tubular neighbourhood of $\Sigma$ in $M$. 
One has $M=\bar V \cup \bar K$, with $\bar K=p(K)$. 
We have to prove that $M$ is a \hs.

As $X$ is a conjugation manifold, the restriction map $H^\ddstar_G(X)\to H^\ddstar_G(X^G)$ is injective by \proref{P.ef} and \remref{R.finGsp}.
Also, $H^\ddstar_G(K)\approx H^\ddstar(\bar K)$ and $H^\ddstar_G(\partial V)\approx H^\ddstar(\partial\bar V)$ since the $G$-action on $K$ and $\partial V$ is free.
The Mayer-Vietoris sequence in equivariant cohomology looks then like
\eqncount
\begin{equation}\label{MV-XG}
0 \to H^\ddstar_G(X) \to H^\ddstar_G(\Sigma)\oplus H^\ddstar(\bar K) \to H^\ddstar(\partial \bar V) \to 0 . 
\end{equation}
Since $\Sigma\subset X$ is codimension $2$, the manifold $K$ is connected.
Therefore, $\bar K$ has a non-trivial $2$-fold cover, which implies that
$b_1(K)\geq 1$. 
Also, since $X$ is a conjugation space, $H^1(X)=0$, so $\dim H^1_G(X)=1$.
As $G$ acts trivially on $\Sigma$, one has $\dim H^1_G(\Sigma)=b_1(\Sigma)+1$.
On the other hand, $b_1(\partial\bar V)\leq b_1(\Sigma)+1$ by the Gysin sequence 
of the circle bundle $\partial\bar V\to\Sigma$. Thus, Sequence~\eqref{MV-XG}
implies that $b_1(\bar K)=1$, $b_1(\partial\bar V)=b_1(\Sigma)+1$ and that
$H^\ddstar(\Sigma)\oplus H^\ddstar(\bar K) \fl{\approx} H^\ddstar(\partial \bar V)$ is an isomorphism. 
This isomorphism sits in the Mayer-Vietoris sequence for $M$ which
implies that $H^1(M)\approx H^3(M)=0$. The decompositions $X=V\cup K$ and 
$M=\bar V\cup \bar K$ give the system of equations
$$
\left\{
\begin{array}{rcl}
\chi(X)&=&  \chi(\Sigma) + \chi(K) -\chi(\partial V) \\
\chi(M) &=&  \chi(\Sigma) + \chi(\bar K) -\chi(\partial \bar V)
\end{array}
\right. .
$$
As $\chi(\partial V)=2\chi(\partial\bar V)$ and $\chi(K)=2\chi(\bar K)$,
one deduces that 
\eqncount
\begin{equation}\label{Eqchi}
\chi(X) + \chi(\Sigma)= 2\chi(M) \, ,
\end{equation}
(compare \cite[Chapter~III, Theorem~7.10]{bredon1}). In our case, Equation~\eqref{Eqchi} amounts to
$$
2 + b_2(X) + 2 - b_1(\Sigma) = 2\chi(M) \, ,
$$
which implies that $\chi(M)=2$. Hence, $M$ is a \hs.
\end{proof}

We now construct the correspondence from \ztk s to conjugation manifolds.

\begin{Lemma}\label{knot2con}
Let $(M,\Sigma)$ be a \ztk. Then $M$ has a unique oriented branched covering $\widehat M\to M$, with branched locus $\Sigma$.
Moreover, $\widehat M$ is a conjugation manifold.
\end{Lemma}

\begin{proof} 
Let $W$ be a closed tubular neighbourhood
of $\Sigma$ in $M$ and let $L$ be the complement
of the interior of $W$. 
For $i=1,2$, one has the ``Alexander duality'
$$
H^i(L)\fl{\approx} H^{i+1}(M,L) \fl{\approx} H^{i+1}(W,\partial W)  \fl{\approx}
H^{i-1}(\Sigma) \ ,
$$
the last arrow being the Thom isomorphism.
Thus, $b_1(L)=1$, $b_2(L)=b_1(\Sigma)$. Also, 
the Mayer-Vietoris exact sequence for the decomposition $M=W\cup L$
gives the isomorphisms
\eqncount
\begin{equation}\label{MV-iso}
H^i(\Sigma)\oplus H^i(L) \fl{\approx} H^i(\partial W)  \hskip 3mm
(i=1,2) \, .
\end{equation}
Therefore, $b_1(\partial W)=b_2(\partial W)=b_1(\Sigma)+1$.
Since $\Sigma$ is of codimension $2$ in $M$, the manifold $L$ is connected.
As $b_1(L)=1$, there is a unique connected $2$-fold cover $\widetilde L\to L$.
The induced cover $\widetilde {\partial W}\to \partial W$ 
is connected:  otherwise, using \lemref{P.Exibran} and \remref{R.Exibran},
 $\widetilde L\to L$ could be extended to a connected $2$-fold covering of $M$, 
contradicting the assumption that $H^1(M)=0$. Hence, $\widetilde L\to L$ extends to a unique
branched covering $\widehat M\to M$ with branched locus $\Sigma$, 
see \lemref{P.Exibran}. The complement
$\widehat W$ of the interior of $\widetilde L$ is a tubular neighbourhood of $\Sigma$
in $\widehat M$.

In equivariant cohomology,
the Mayer-Vietoris exact sequence for $\widehat M=\widehat W \cup \widetilde L$ starts as
\eqncount
\begin{equation}\label{MV-YG}
0 \to H^1_G(\widehat M) \to H^1_G(\Sigma)\oplus H^1(L) \to H^1(\partial W)
\end{equation}
with $H^1_G(\Sigma)=H^1(\Sigma)\oplus\bbz_2$. Using the isomorphism of \eqref{MV-iso}
for $i=1$, we deduce that $H^1_G(\widehat M)=\bbz_2$. This implies that $H^1(\widehat M)=0$.
Indeed, the choice of a $G$-fixed point in $\widehat M$ provides a section to the
fibration $\widehat M\to \widehat M_G\to BG$. 
The homomorphism $H^\ddstar(BG)\to H^\ddstar_G(\widehat M)$ is then injective and  
the Serre spectral sequence for the fibration $\widehat M\to \widehat M_G\to BG$
gives the exact sequence
$$
0 \to H^1(BG)\to H^1_G(\widehat M) \to H^0(BG;H^1(\widehat M)) \to 0 \, .
$$
Thus, if $H^1_G(\widehat M)=\bbz_2$, then $0=H^0(BG;H^1(\widehat M))=H^1(\widehat M)^G$.
But, a finite dimensional $G$-vector space $V$ over $\bbz_2$ vanishes if $V^G=0$
(if $0\neq v \neq \tau(v)$ then $0\neq v+ \tau(v)\in V^G$). Therefore, $H^1(\widehat M)=0$.
By Poincar\'e duality, we have then $H^{odd}(\widehat M)=0$.
Equation~\eqref{Eqchi} holds with the same proof and gives
$\chi(\widehat M) + \chi(\Sigma)= 2\chi(M)=4$ which implies that $b_2(\widehat M)=b_1(\sigma)$.
Therefore, $\widehat M$ is a conjugation manifold by \proref{P.Cma4n}.
\end{proof}

\begin{proof}[The proof of Theorem A] 
With the equivalence relations used for the statement of Theorem~A,
let $\calc$ be the set of equivalence classes of oriented conjugation $4$-manifolds
and let $\caln$ be that of classes of \ztk s.
By \lemref{con2knot}, the correspondence $X\mapsto X/G$ associates a a \ztk\ to a
conjugation $4$-manifold. 
By \lemref{smoothquoUNI}, this correspondence produces a well defined
map $\Phi\:\calc\to\caln$.

By \lemref{knot2con}, the correspondence $M\mapsto \widehat M$ sends a
\ztk\ to an oriented conjugation $4$-manifold. 
By \lemref{L.Unibran}, this provides a well defined map 
$\Psi\:\caln\to\calc$.

The fact that $\Phi\pcirc\Psi={\rm id}_{\caln}$ is guaranteed by \lemref{brcov2invo}.
That $\Psi\pcirc\Phi={\rm id}_{\calc}$ follows from \lemref{smoothquo}
(since $p\:X\to X/G$ is a branched covering with branched locus $\Sigma$), 
and the uniqueness part of \lemref{P.Exibran}.
\end{proof}

\section{Examples and Remarks}\label{S.Exples}

We present  some examples and remarks concerning the bijection of Theorem A.

\begin{example}\label{E.TrivKnot}
Let $\Sigma=S^2\subset \bbr^3\times 0 \subset\bbr^3\times\bbc$. The involution on 
$X=S^4\subset\bbr^3\times\bbc$ induced by the linear map $(w,z)\mapsto (w,-z)$
is a conjugation by \proref{L.conjSn}. The map $q\:\bbr^3\times\bbc\onto \bbr^3\times\bbc$
given by $q(w,z)=(w,z^2)$ identifies $(\bbr^3\times\bbc)/G$ with $\bbr^3\times\bbc$ and
$S^4/G$ with $S^4$. This shows that, with the smooth structure of \lemref{smoothquo},
$S^4/G$ is diffeomorphic to $S^4$. 
The image in $X/G$ of the sphere $S^3\subset\bbr^3\times\bbr$ is a $3$-disk with
boundary $\Sigma$. Hence, under the bijection $X\mapsto X/G$ of Theorem~A, the standard conjugation
sphere corresponds to the trivial knot. 
\end{example}

Interesting examples occur with $\Sigma=\rp$.

\begin{Proposition}\label{P.rpor}
Let $(M,\Sigma)$ be a \ztk\ with $\Sigma=\rp$. Then, any homeomorphism
$h\:(M,\Sigma)\to (M,\Sigma)$ is of degree one. In consequence,
$(M,\sigma)$ and $(-M,\sigma)$ are inequivalent \ztk s.
\end{Proposition}

\begin{proof}
By the uniqueness of branched coverings \cite[Prop.~3]{lines1}, $h$
is covered by a homeomorphism $\widehat h\:\widehat M\to \widehat M$.
The cohomology ring  $H^\ddstar(\widehat M)$ is isomorphic to that of $\cp$.
By the universal coefficient theorem and Poincar\'e duality,
$H^\ddstar(\widehat M;\bbq)\approx \bbq[a]/(a^3)$, with $a\in H^2(\widehat M;\bbq)$.
Thus, $\widehat h^*(a)=\lambda a$ for some $\lambda\in\bbq$ 
and $\widehat h(a^2)=\lambda^2 a^2$. 
Therefore, $\widehat h$ is of degree $1$ and so is $h$. 
\end{proof}

\begin{example}\label{E.rp1}
The projective space $X=\cp$ with the complex conjugation is a conjugation manifold with $X^G=\rp$.
The quotient $X/G$ is diffeomorphic to $S^4$ (see \proref{classic1} below),
and $(X/G,\rp)$ is the ``standard" embedding of $\rp$ into $S^4$ (see Lawson \cite{lawson2} for an explicit description).
Let $r\:X\to X$ be the diffeomorphism given by 
$r(x,y,z)=(x,y,-z)$. It commutes with the complex conjugation and thus 
descends to an involution $\bar r$ of $X/G\approx S^4$, preserving $\rp$.
The diffeomorphism $r$ is isotopic to the identity by the isotopy
$r_t(x,y,z)=(x,y,e^{i\pi t}z)$.
Thus, $r$ and $\bar r$ are of degree $1$, in accordance with \proref{P.rpor}.
\end{example}

\section{Applications to knots in $S^4$}\label{S.pi1}

The $2$-fold branched coverings in which the  branch locus is a knotted $2$-sphere in $S^4$,  are particularly interesting, We first investigate the relation between the fundamental groups of a conjugation $4$-manifold $X$ and its quotient $X/G$.

Let $(M,\Sigma)$ be a \ztk, and let $W$ be a closed tubular neighbourhood
of $\Sigma$. Then $\partial W \to \Sigma$ is a locally trivial
$S^1$-bundle.
We get an exact sequence
$$
C_\infty \to \pi_1(\partial W) \to \pi_1(\Sigma) \to 1 \, ,
$$
where $C_\infty$ is an infinite cyclic group.
The image of a generator  $\gamma \in C_\infty$ via the composed homomorphism
$C_\infty \to \pi_1(\partial W)\to \pi_1(M-\Sigma)$ is called a
{\it meridian}, and denoted $m \in\pi_1(M-\Sigma)$ . A \ztk\ admits two meridians, which are  inverses  of each other.
From the proof of \lemref{knot2con} we see that $H^1(M-\Sigma)\cong \bbz_2$,
so there is a unique epimorphism $\phi\:\pi_1(M-\Sigma)\to\bbz_2$. Furthermore, $\phi(m) \neq 0$.

\begin{Proposition}\label{1compPi1}
Let $(M,\Sigma)$ be a \ztk\ and let $\widehat M$ be the associated conjugation $4$-manifold. 
If $m\in\pi_1(M-\Sigma)$ is a meridian, then
$$
\pi_1(\widehat M) \cong \ker\phi/ \la m^2\ra
$$
\end{Proposition}

\begin{proof}  Let $M = W \cup L$, where $L \simeq M - \Sigma$, and similarly $\widehat M = \widehat W \cup \widetilde L$. As noted above, $\pi_1(\partial W) \to \pi_1(\Sigma)$ is surjective, and so is $\pi_1(\widetilde{\partial W}) \to \pi_1(\Sigma)$ by the same argument. 
We also have the description $\pi_1(\widetilde L) = \ker\phi$.
Let $\gamma$ be a generator of $C_\infty$ sent to $m \in \pi_1(L)$ and let 
$\bar\gamma$ be the image of $\gamma$ in $\ker(\pi_1(\partial W) \to \pi_1(\Sigma))$.
Then $\ker(\pi_1(\widetilde{\partial W}) \to \pi_1(\Sigma))$ is generated by $\gamma^2$. \proref{1compPi1} follows from
the Van Kampen theorem, 
because of the surjectivity of the map $\pi_1(\widetilde{\partial W}) \to \pi_1(\Sigma)$.
\end{proof}

\begin{Corollary}\label{1conncon}
Let $(M,\Sigma)$ be a \ztk\ and let $m\in\pi_1(M-\Sigma)$ be a meridian.
The associated conjugation $4$-manifold
$\widehat M$ is simply connected if and only if $\ker\phi$
is the normal closure of $m^2$.
\qedhere
\end{Corollary}

\begin{Proposition}\label{1conncon-bis}
Let $X$ be a conjugation $4$-manifold. If $X$ is simply connected,
then $X/G$ is homeomorphic to $S^4$.
\end{Proposition}

\begin{proof}
Let $M=X/G$ and $\Sigma=X^G$. We know that $(M,\Sigma)$ is a
\ztk. Let $m$ be a meridian for $(M,\Sigma)$. By \corref{1conncon},
$\ker\phi$ is the normal closure of $m^2$. But $\phi(m)\neq 0$,
as seen in the proof of \lemref{knot2con} since the $2$-fold covering
given by $\phi$ is not trivial over $\partial W$. This implies that
$\pi_1(M-\Sigma)$ is the normal closure of $m$. As above, the Van Kampen theorem
implies that $M$ is simply connected. By the universal coefficient
theorem and Poincar\'e duality, a simply connected \hs\ is an integral homology sphere.
Therefore, $M$ is homotopy equivalent to $S^4$, and hence homeomorphic
to $S^4$ by Freedman's proof \cite{freedman1} of  the Poincar\'e conjecture in dimension $4$.
\end{proof}

As mentioned in the Introduction, it is well-known that the classical conjugation $4$-manifolds all have quotient the standard smooth $S^4$. For the reader's convenience, we include a proof  using results on group actions due to Bredon and Orlik-Raymond (but not the deep results of Freedman \cite{freedman1} or Cerf \cite{cerf1}).
\begin{Proposition}[Arnold,  Kuiper, Massey]\label{classic1} Let $(X,G)$ denote the classical conjugation $4$-manifolds (i)  $\CP^2$ with complex conjugation,  or (ii) $S^2 \times S^2$ with the complex conjugation in each factor. Then the quotient $X/G$ with the smooth structure given by \lemref{smoothquo} is diffeomorphic to the standard smooth $S^4$.
\end{Proposition}
\begin{proof} There is a smooth $SO(3)$-action on $X=\CP^2$ which commutes with complex conjugation. By \lemref{smoothquo-equiv}, the quotient space $X/G$ inherits the structure of a smooth $SO(3)$-manifold. The classification of smooth cohomogeneity one actions of $SO(3)$ by Bredon \cite[Theorem VI.6.3]{bredon1} shows that $(X/G, SO(3))$ is $SO(3)$-equivariantly diffeomorphic to the standard $SO(3)$-action on $S^4$. Therefore $X/G$ is diffeomorphic to $S^4$. The case $X=S^2\times S^2$ is done in Example \ref{E.s2xs2} and in Example \ref{classic2}.
\end{proof}

To describe other examples for \proref{1conncon-bis}, we start with a technique of Mazur and Zeeman \cite{zeeman1}.
Let $N$ be a smooth oriented $3$-dimensional $\bbz_2$-homology sphere. Let $h\:N\to N$ be
a diffeomorphism such that
\begin{enumerate}
\item $h$ preserves the orientation.  
\item $h$ has a fixed point.
\end{enumerate}
Let $T_h=N\times [0,1]/\{(x,1)\sim (h(x),0)\}$ 
be the mapping torus of $h$. Choose a fixed point $x_0\in N$ for $h$. 
The map $\psi_0\:[0,1]\to T_h$ given by $\psi_0(t)=[x_0,t]$ is a parametrisation
of a circle $S$ in $T_h$. The normal bundle of $S$ is trivial by condition (1).
We can then choose a parametrisation 
$\psi\:[0,1]\times D^3\to T_h$ of a tubular neighbourhood of $S$ extending $\psi_0$. 
We consider the
the surgery using $\psi$, producing a smooth $4$-manifold
\eqncount
\begin{equation}\label{Eq-surg}
 M_{h,\psi} = [T_h - \psi([0,1]\times {\rm int\, } D^3)] \, \cup_{\dot\psi} D^2\times S^2 \, ,
\end{equation}
where $\dot\psi$ is the restriction of $\psi$ to $[0,1]\times S^2$.
The manifold $M_{h,\psi}$ contains $\{0\}\times S^2$, so we get a pair
 $(M_{h,\psi},S^2)$.

\begin{Lemma}\label{L.surgery} \textcolor{white}{.}
\begin{enumerate}
\item $(M_{h,\psi},S^2)$ is a \ztk.
\item If $N$ is a $\bbz$-homology sphere, then $M_{h,\psi}$ is a $\bbz$-homology sphere.
\item The fundamental group of $M_{h,\psi}$ is isomorphic to 
the quotient of $\pi_1(N)$ by the relations $x=h_*(X)$ for all $x\in \pi_1(N)$.
\item The fundamental group of the associated conjugation $4$-manifold $\widehat M_{h,\psi}$ 
is isomorphic to the quotient of $\pi_1(N)$ by the relations 
$x=h^2_*(X)$ for all $x\in \pi_1(N)$, where $h^2=h\pcirc h$.
\end{enumerate}
\end{Lemma}

\begin{proof}
By the Serre spectral sequence of the bundle $N\to T_h\to S^1$, the homology of $T_h$ is isomorphic
to that of $N\times S^1$. Conclusions (1) and (2) then follow from the Mayer-Vietoris
sequence of the decomposition~\eqref{Eq-surg}. 

For the fundamental group, choose a base point $\tilde y_0\in S^2$ and let 
$y_0=\dot\psi(\tilde y_0,0)\in N\subset T_h$. The map $\psi_1(t)=\dot\psi(\tilde y_0,t)$
represents an element $m\in \pi_1(T_h,y_0)$ and the fundamental group
of $T_h$ is the HNN-extension 
$$
\pi_1(T_h,y_0) \approx \ \langle\, \pi_1(N), m \,;\,  
mxm^\mun= h_*(x), \, x\in \pi_1(N)\,\rangle 
$$
(we use the notations of \cite[\S~IV.2]{lyndon-schupp1}). By the Van Kampen theorem applied
to the decomposition~\eqref{Eq-surg}, 
$\pi_1(M_{h,\psi},y_0)$ is the quotient of $\pi_1(T_h,y_0)$ by the normal closure of $m$.
This proves (3). 

To prove (4), let $L= M_{h,\psi} - ({\rm int\,}D^2\times S^2)=
T_h - \psi([0,1]\times {\rm int\, } D^3)$. 
The element $m$ is a meridian and 
the epimorphism $\phi\:\pi_1(M_{h,\psi}-S^2,y_0)\to \bbz_2$ sends $m$ to the generator
and $\pi_1(N)$ to $0$. Thus, $\widetilde L$ is the mapping torus of $h^2$ and there is
a decomposition
\eqncount
\begin{equation}\label{hat-remark}
\widehat M_{h,\psi}= M_{h^2, \hat\psi} =  T_{h^2}\cup D^2\times S^2
\end{equation}
for some parametrisation $\hat\psi$ analogous to~\eqref{Eq-surg}.
Conclusion~(4) follows from Van Kampen's theorem.
\end{proof}

Examples of $3$-dimensional $\bbz_2$-homology spheres are given by cyclic branched coverings of classical
knots. In such case, following results of Zeeman
\cite[Corollary~1, p.~486]{zeeman1}
and Gordon \cite[Theorem~3.1]{gordon2}, Pao proved the following result
(see \cite[\S~3]{pao1}).

\begin{Lemma}\label{L.surgery2}
Let $N$ be the $p$-fold cyclic covering $S^3$ branched over a knot. 
Let $h\:N\to N$ be the diffeomorphism corresponding to the action of a generator of $\bbz_p$.
Then $M_{h,\psi}$ is diffeomorphic to $S^4$ for any $\psi$.  \mancqfd
\end{Lemma}

\begin{example}\label{Ex.lensSp2} 
Let $N=L(p,q)$ be a $3$-dimensional lens space with $p$ odd. 
By \cite[Satz~6]{schubert1}, $N$ is the $2$-fold branched covering
of $S^3$ for some knot. By \lemref{L.surgery2}, this gives an involution $h$ of $N$
for which $M_{h,\psi}$ is diffeomorphic to $S^4$ (but $\pi_1(\widehat M_{h,\psi})\approx \bbz_p$).
\end{example}

\begin{example}\label{Ex.Delta2} 
Let $N$ be the Poincar\'e homology sphere
$$
N = SO(3)/A_5
$$ 
with $\pi_1(N)=\Delta$ the binary icosahedral group, defined as the universal central
extension of $A_5$. The conjugation in $SO(3)$ by an element $b\in A_5$ produces
a diffeomorphism $h\:N\to N$. The induced automorphism $h_*$ of $\Delta$ is the
conjugation by an element $\tilde b$ over $b$. 
\lemref{L.surgery} has the following consequences on the manifolds
$M_{h,\psi}$ and $\widehat M_{h,\psi}$ (for any choice of $\psi$).
\begin{itemize}
 \item $M_{h,\psi}$ is an integral homology sphere.
\item If $b$ is not trivial, $M_{h,\psi}$ is a homotopy sphere.
\item If $b^2$ is not trivial, $\widehat M_{h,\psi}$ is a homotopy sphere.
Indeed, the relations $x=b^2xb^{-2}$ kills $A_5$ which is simple,
so the relations $y=\tilde b^2y\tilde b^{-2}$ kills $\Delta$. 
\end{itemize}
It is classical that $N=SO(3)/A_5$ is the $5$-fold covering of $S^3$ branched over the trefoil
knot \cite[\S~65]{seifert-threlfall}. By \lemref{L.surgery2} and (\ref{hat-remark}), 
$M_{h,\psi}$ and $\widehat M_{h,\psi}$ are both diffeomorphic to $S^4$
for some choice of $b$ of order $5$.
\end{example}

\begin{remark} 
\exref{Ex.Delta2} with $b^2\neq 1$ produces counter-examples to the generalised
Smith conjecture in dimension $4$: a smooth involution on a homotopy $4$-sphere
with a non-trivial knot (even topologically) as fixed point set. 
The first such examples appeared in 1966 in the work of Giffen \cite[Theorem 3.3]{giffen1}. Later Gordon \cite{gordon1}, \cite{gordon2} and
Sumners \cite{sumners1} proved that there are infinitely many 
non-equivalent knots in $S^4$ which are the fixed point sets of smooth involutions. In addition, 
Cameron Gordon (email communication) has informed us that the quotient spaces $S^4/G$ of his examples
are  obtained from $S^4$ by a ``Gluck surgery" construction on a (twist-spun) knot, and hence are diffeomorphic to $S^4$ by \cite[Theorem 3.1]{gordon2}.
\proref{L.conjSn} now  implies the following statement:
\end{remark}

\begin{Proposition}\label{P.incconS4}
There are infinitely many smooth conjugations on $S^4$ which are topologically inequivalent.
\mancqfd 
\end{Proposition}

\begin{remark}
Further examples of smooth conjugation $4$-manifolds might also come from the \emph{rim surgery} construction of  Fintushel and Stern \cite{fintushel-stern98}, \cite{fintushel-stern-sunukjian1}, although it is not clear at present how to detect exotic smooth structures on the $2$-fold branched coverings of $\bbz_2$-homology $4$-spheres. 
\end{remark}

\section{Conjugations on topological $4$-manifolds}\label{S.top}

We first state a topological version of Theorem A. The involution in a topological conjugation
manifold is supposed to be \textit{locally linear} (also called \textit{locally smooth} 
in \cite[Chapter~IV]{bredon1}). We also consider topological
locally flat
\ztk s. Two of them, $(M,\Sigma)$ and $(M',\Sigma')$ are {\it topologically equivalent}
if there is an orientation-preserving homeomorphism $h\:M\to M'$
such that $h(\Sigma)=\Sigma'$.

\begin{thma-top}
The correspondence $X\mapsto (X/G,X^G)$ defines a bijection between
\renewcommand{\labelenumi}{(\alph{enumi})}
\begin{enumerate}
\item the orientation-preserving $G$-homeomorphism classes of oriented
connected topological conjugation $4$-manifolds,
and
\item
the topological equivalence  classes of topological locally flat \ztk s.
\end{enumerate}
\end{thma-top}

\begin{proof}
The orbit space $X/G$ of a locally linear action is a closed topological $4$-manifold, and the image $p(X^G)$ of the fixed set is a locally flat submanifold of $X/G$. By \cite[\S~9.3]{freedman-quinn1}, this submanifold admits a normal bundle and therefore a tubular neighbourhood. Lifting this tubular neighbourhood to $X$ gives a $G$-invariant tubular neighbourhood for $X^G$ in $X$. Similarly, a locally flat submanifold $\Sigma$ of $M$ admits a tubular neighbourhood.
The proofs of Lemmas~\ref{con2knot} and~\ref{knot2con} can be carried out using these tubes (using \remref{R.finGsp} as noted to justify the application of \proref{P.ef}).
The existence of the $2$-fold branched covering 
$\widehat M\to M$ is guaranteed by \cite[Prop.~2]{lines1}.

The arguments of the proof of Theorem~A (end of \secref{S.proofThmA}) are much simpler
than in the smooth case. They come from the fact that the constructions
under consideration are functorial for homeomorphisms.
For $X\mapsto X/G$, this is obvious. For $M\mapsto \widehat M$, this follows from
\cite[Prop.~3]{lines1}.
\end{proof}

\begin{Examples} \rm
Finashin, Kreck and Viro  \cite{finashin-kreck-viro1} have constructed an infinite family of
topologically equivalent, but smoothly inequivalent embeddings of
$\Sigma= \Sharp_{10} \RP^2$ in $S^4$.
In all these examples, the fundamental group of the complement is just $\bbz/2$.
By Theorem~A, this gives an infinite family of smooth conjugation $4$-manifolds
which are topologically equivalent by Theorem B but not diffeomorphic (since the associated conjugation $4$-manifolds are non-diffeomorphic  Dolgachev surfaces).

Other examples, this time of topologically equivalent, but smoothly inequivalent knotted surfaces in $\CP^2$, were constructed by Finashin \cite{finashin1}.
\end{Examples}

\begin{remark} 
Non-oriented versions of Theorems A and B hold: just leave out
the words ``orientation preserving'' and ``oriented'' in (a) and in the
definition of \ztk s. For instance, $(M,\Sigma)$ is equivalent to
$(-M,\Sigma)$. This is definitely a coarser equivalence relation, as seen
in \proref{P.rpor}. 
\end{remark}

\section{Appendix: Branched coverings and smooth structures}\label{S.bran}

Branched covering spaces of manifolds is a classical topic in geometric topology, which appears frequently in the literature (for example, see the references cited in Durfee and Kauffman \cite{durfee-kauffman1} and Lines \cite{lines1}). By a \emph{$2$-fold branched covering} we mean a ramified $2$-fold covering in which the branch locus is a closed submanifold of codimension two. 

In the proof of Theorem A, we need to know the relationship between smooth structures on the total space and quotient space of a $2$-fold branched covering of $4$-manifolds. This material may be well-known, but we were not able to find the right references for our proofs. 
Smooth manifolds are so strikingly different from topological manifolds in dimension $4$, that these issues perhaps deserve some extra attention.  This section is included to provide a detailed account, as a service to the reader.

\smallskip
The map $z\mapsto z^2$ from $\bbc$ to $\bbc$ is the simplest example of a 
$2$-fold branched covering, with $\{0\}$ being the (strict) branched locus.
We will need a precise local description of this example. Let $D$ be the unit disk in $\bbc$. 
Identify $SO(2)$ with $S^1$ and $O(2)$ with the $\bbr$-linear isometries of $\bbc$.
The homomorphism $\gamma\mapsto\gamma^2$ of $S^1$ extends to a smooth epimorphism
\eqncount
\begin{equation}\label{E-defpsi}
 \psi\:O(2)\to O(2) \, .
\end{equation}
Let $P\to K$ be a smooth principal $O(2)$-bundle.
Consider the two Borel constructions
$$
P\times_{O(2)} D = P\times D \big/ \{(a\,\alpha,z)=(a,\alpha\,z)\mid \alpha\in O(2)\}
$$
and 
$$
P\bar\times_{O(2)} D = P\times D \big/ 
\{(a\,\alpha,z)=(a,\psi(\alpha)\,z)\mid \alpha\in O(2)\} \, .
$$
The map $(a,z)\mapsto (a,z^2)$  descends to a smooth surjection
\eqncount
\begin{equation}\label{E-bran}
 q\: P\times_{O(2)} D \to P\bar \times_{O(2)} D 
\end{equation}
which will be our local model for a $2$-fold  branched covering with branched locus $K$.
The general definition is the following.

\begin{definition}\label{def-bran} 
Let $M$ be a smooth manifold of dimension $n$ with a codimension $2$ submanifold $N$. A smooth map $p\: (X,Y) \to (M,N)$ is a ($2$-fold) \textit{branched covering}
with \textit{branched locus} $N$  if 
\begin{itemize}
\item $p\:X- Y\to M-N$ is a smooth $2$-fold covering, in particular a local diffeomorphism.
\item $p\: Y\to N$ is a diffeomorphism. We often identify $Y$ with $N$ via $p$.
\item there are closed tubular neighbourhoods $Y\subset \tilde V\subset X$ and $N\subset V\subset M$
for $Y$ and $N$ such that $p(\tilde V)=V$ and $p|\tilde V$ has, up to diffeomorphism,
the form of~\eqref{E-bran}.
\end{itemize}
\end{definition}

This definition may be compared with the properties of smooth branched coverings 
given in Durfee and Kauffman \cite[Prop.~1.1]{durfee-kauffman1}. On one hand, it is simpler
because we are dealing with the special case of $2$-fold coverings. On the other hand, Definition~\ref{def-bran}
is more precise: we specify  the model around the branched locus (compare  
part~(ii) of Proposition~1.1 in~\cite{durfee-kauffman1}).

\begin{ccote}\textbf{Quotient structure}.
Our first task is to show how a smooth $G$-action on a closed manifold $X$ determines 
a smooth structure on the quotient space $X/G$, which is unique up to diffeomorphism. 
\begin{Lemma}\label{smoothquo}
Let $X$ be a smooth  $G$-manifold such that  the fixed point set $X^G$ is 
a closed manifold of codimension $2$. 
Then, 
\begin{enumerate}
\item $X/G$ admits the structure of a smooth manifold such that
$p\:X\to X/G$ is a branched covering with branched locus $p(X^G)$.
\item if $X$ is a closed manifold, any two such structures on $X/G$ are diffeomorphic. 
\end{enumerate}
\end{Lemma}

\begin{proof} 
Let $X_{free}=X-X^G$. 
The quotient map $p_{free}\:X_{free} \to X_{free}/G$ is a covering projection, and hence $X_{free}/G$
has a unique a smooth structure such that $p_{free}$ is a local diffeomorphism.
We shall put a smooth structure on a neighbourhood of $X^G$ in $X/G$
which agrees with that on $X_{free}$.

\st{1.  Existence} To a $G$-invariant Riemannian metric $g$ on $X$, we will  associate a smooth
structure $(X/G)_g$ on $X/G$ satisfying condition (1) of \lemref{smoothquo}.
Let $\nu$ be the normal bundle to $X^G$ in $X$ given by the metric $g$,
so $\nu_x= (T_xX^G)^\perp \subset T_xX$ for $x\in X^G$.
Let $P\to X^G$ be the 
$O(2)$-principal bundle associated to $\nu$, so $P_x$ is the space of orthonormal
frames in $\nu_x$. The space $P\times_{O(2)} D$ 
is a smooth $G$-manifold,  with the involution $\tau(a,z)=(a,-z)$.
It is $G$-diffeomorphic to the unit disk bundle associated to $\nu$

Let us perform the  equivariant tubular neighbourhood 
construction \cite[Chapter~VI, Theorem~2.2]{bredon1} using the exponential map for
the metric $g$. We say that $g$ is \textit{calibrated around $X^G$} if  
the exponential map is an embedding on $P\times_{O(2)} D$.
As $X^G$ is compact, one can multiply $g$ by a constant (scaling)
so that it is calibrated. Therefore, there exists a $G$-invariant 
neighbourhood $V$ of $X^G$ in $X$ and a $G$-equivariant diffeomorphism
$$
\varphi\: P\times_{O(2)} D \fl{\approx} V  \, .
$$ 
The map $q$ of~\eqref{E-bran} sits in a commutative diagram
$$
\begin{array}{c}{\xymatrix@C-3pt@M+2pt@R-4pt{%
 & P\times_{O(2)} D \ar@{>>}[dl]_(0.50){\pi}
\ar[dr]^(0.50){q}  
\\
P\times_{O(2)} D\big/G \ar[rr]^(0.50){h}_(0.50){\approx}  &&
P\bar\times_{O(2)} D
}}\end{array} \ .
$$ 
The map $h$ is a continuous bijection between compact spaces, and hence a homeomorphism.
Therefore, we get a commutative diagram
\eqncount
 \begin{equation}\label{E-smoothquo-20}
\begin{array}{c}{\xymatrix@C-3pt@M+2pt@R-4pt{%
P\times_{O(2)} D \ar[d]^(0.45){q}
\ar[r]^(0.59){\varphi}_(0.59){\approx}  &
V \ar[d]^(0.50){p}  \\
P\bar\times_{O(2)}  D\ar[r]^(0.53){\bar\varphi}_(0.53){\approx}  &
V/G
}}\end{array}
\end{equation}
where $\bar\varphi$ is a homeomorphism. As $q$ is smooth, the homeomorphism $\bar\varphi$
provides a smooth structure on $V/G$, which is a neighbourhood of $X^G$ 
in $X/G$, and the projection $p\:V\to V/G$ is smooth.
As $\varphi$ composed with the inclusion $V\hookrightarrow X$ is a smooth embedding,
the smooth structures on $X_{free}/G$ and on $V/G$ agree on $V/G-X^G$. 
This defines the structure $(X/G)_g$, which does not depends on the scaling.
By diagram~\eqref{E-smoothquo-20}, it satisfies condition~(1) of \lemref{smoothquo}.

\st{2. Metric independence} If $g$ and $g'$ are two $G$-equivariant Riemannian metrics on $X$, and if $X$
is a closed manifold,
then we will show that $(X/G)_g$ and  $(X/G)_{g'}$ are diffeomorphic. 
The family $tg' + (1-t)g$ ($t\in [0,1]$) defines a $G$-invariant Riemannian metric 
$\check g$ on the
manifold with boundary $L=X\times [0,1]$. By scaling $\check g$,
we may suppose that it is calibrated around $L^G$. 
The above construction provides a smooth structure $(L/G)_{\check g}$ (the presence
of boundaries does not create difficulties). Using diagram~\eqref{E-smoothquo-20},
one shows that the projection $b\:(L/G)_{\check g}\to [0,1]$ is a submersion.
As $X/G$ is a closed manifold, integrating a gradient-like vector field for $b$
provides a diffeomorphism between $b^\mun(0)=(X/G)_g$ and $b^\mun(1)=(X/G)_{g'}$.

\st{3. Uniqueness} 
Recall that all smooth structures on $X/G$ satisfying condition (1) of \lemref{smoothquo} will agree on
$X/G-X^G$. We have to show that this is the case  around $X^G$.

Suppose that $X/G$ is endowed with a smooth structure such that 
$p\:X\to X/G$ is  a $2$-fold branched covering with branched locus $p(X^G)=X^G$.
By definition, there is a neighbourhood $W$ of $X^G$ in $X$ and 
a commutative diagram
\eqncount
 \begin{equation}\label{E-smoothquo-40}
\begin{array}{c}{\xymatrix@C-3pt@M+2pt@R-4pt{%
P\times_{O(2)} D \ar[d]^(0.45){q}
\ar[r]^(0.59){\varphi}_(0.59){\approx}  &
W \ar[d]^(0.50){p}  \\
P\bar\times_{O(2)}  D\ar[r]^(0.53){\bar\varphi}_(0.53){\approx}  &
W/G
}}\end{array}
\end{equation}
where $\varphi$ and $\bar\varphi$ are smooth embeddings. If we restrict
these embeddings to $P\times_{O(2)} D^\ddstar$ and $P\bar\times_{O(2)} D^\ddstar$,
where $D^\ddstar=D-\{0\}$, then diagram~\eqref{E-smoothquo-40} is a morphism
of $2$-fold (unbranched) covering spaces. The deck transformation
on $P\times_{O(2)} D^\ddstar$ is given by $(a,z)\mapsto (a,-z)$. Hence,
$\varphi$ is a a $G$-equivariant embedding. 

Endow $P\times D$ with a Riemannian metric which is the product of
a $O(2)$-invariant Riemannian metric on $P$ with the standard metric on $D$.
This descends to a Riemannian metric on $P\times_{O(2)} D$, which is $G$-invariant.
One can construct a $G$-invariant Riemannian metric $g$ on $X$ so that
$\varphi$ is an isometry. Hence, $\varphi$ is actually the normal exponential map
and diagram~\eqref{E-smoothquo-40} plays the role of diagram~\eqref{E-smoothquo-20}
to define the smooth structure $(X/G)_g$ around $X/G$, via the homeomorphism $\bar\varphi$.
As $\bar\varphi$ is a smooth embedding, the given smooth structure on $X/G$ 
coincides with $(X/G)_g$.

Now, let $(X/G)'$ and  $(X/G)''$ be two smooth structures satisfying condition (1).
By the above argument, there are  $G$-invariant Riemannian metrics $g'$ and $g''$
on $X$ such that  $(X/G)'=(X/G)_{g'}$ and $(X/G)''=(X/G)_{g''}$. But
$(X/G)_{g'}$ and $(X/G)_{g''}$ are diffeomorphic, as seen in Step~2.
\end{proof}
\end{ccote}

\begin{Example}\label{E.s2xs2}\rm
Let us consider $S^2\subset\bbc\times\bbr$ with the involution $(z,t)\mapsto (z,-t)$
and let $X=S^2\times S^2$ endowed with the diagonal involution $\tau$. One can construct a smooth
$2$-fold branched covering $\pi\:X\to S^4$ by explicit formulas,  which descends to a homeomorphism
$\bar\pi\:X/G\fl{\approx} S^4$.  We remark that an application of \lemref{smoothquo} then gives the well-known result that $X/G$, with the smooth structure
of \lemref{smoothquo}, is diffeomorphic to $S^4$.

Our coordinates on $X$ will be $(w_1,w_2)$, where $w_j=(r_je^{i\theta_j},t_j)$ for $j=1,2$.
The standard $(S^1\times S^1)$-action on $X$ is defined by 
$$(z_1, z_2)\cdot (w_1,w_2) = (z_1r_1e^{i\theta_1},t_1, z_2r_2e^{i\theta_2},t_2)$$
for all 
$(z_1, z_2) \in S^1 \times S^1$.
Let $\tilde\pi\:X\to \bbc\times\bbc\times\bbr$ be defined by
\eqncount
\begin{equation}\label{s2xs2-eq10}
\tilde\pi(w_1,w_2) =
\big (r_1^2e^{i\theta_1},r_2^2e^{i\theta_2},t_1t_2\big ) \, .
\end{equation}
The function $L=L(w_1,w_2)=\|\tilde\pi(w_1,w_2)\|^2 = r_1^4 + r_2^4 + (t_1t_2)^2$ never vanishes,
so
$$
\pi(w_1,w_2)= \frac{1}{\sqrt{L}}\tilde\pi(w_1,w_2)
$$
defines a smooth map $\pi\:X\to S^4$, which is $(S^1\times S^1)$-equivariant.
One checks that $\pi$ descends to a homeomorphism $\bar\pi\:X/G\fl{\approx} S^4$.

It remains to show that $\pi\colon X \to S^4$ is a branched covering.
Let us consider the following diagram
\eqncount
\begin{equation}\label{s2xs2-eq20}
\begin{array}{c}{\xymatrix@C-3pt@M+2pt@R-6pt{%
X \ar[d]^(0.50){f}
\ar@{>>}[r]^(0.50){\pi}  &
S^4 \ar[d]^(0.50){\bar f}  \\
\bbc \ar[r]^(0.50){\hat\pi}  &
\bbc
}}\end{array}
\end{equation}
where
$$
f(w_1,w_2)=\frac{1}{\sqrt[4]{L}} (t_1+it_2) \quad , \quad
\bar f (\rho_1 e^{i\theta_1},\rho_2 e^{i\theta_2},t) =
\rho_2-\rho_1 + 2it
$$
and $\hat\pi(z)=z^2$. As $r_j^2+t_j^2=1$, one has $t_1^2-t_2^2=r_2^2-r_1^2$ and
diagram~\eqref{s2xs2-eq20} commutes. Observe that $f$ and $\bar f$ are
$(S^1\times S^1)$-invariant.

Derivative computations show that $0\in\bbc$ is a regular value for $f$.
This produces a  $(S^1\times S^1)$-invariant trivialization of the normal bundle
to $X^G=f^\mun(0)$, since $X^G$ is a free $(S^1\times S^1)$-orbit. Also, $0\in\bbc$ is a regular value for $\bar f$: it is easy to find a smooth local section of $\bar f$
into the $(S^1\times S^1)$-slice $\theta_1=\theta_2=0$. This again produces a  $(S^1\times S^1)$-invariant trivialization of  the normal bundle
to $N=\pi(X^G)=\sqrt{2}(e^{i\theta_1},e^{i\theta_2},0)$.

Let $D'$ be a small disk around $0$ in the image of $f$ and let $D''=\hat\pi(D')$.
Using homotheties from $D$ to $D'$ and $D''$ together with the above trivializations
permits us to put  the map $\pi$ into the form~\eqref{E-bran} locally around $X^G$.
\end{Example}

\begin{Remark}\label{R.nonfunct}
\rm 
Statement (2) of \lemref{smoothquo} does not say that the diffeomorphism type of $X/G$ is
functorial. If  $h\:X\to X'$ is a $G$-equivariant diffeomorphism, the induced homeomorphism
$\bar h\:X/G\to X'/G$ is in general not smooth. For example, take the
standard involution $(u,v)\mapsto (-u, -v)$ in $\bbr^2$ and the map $h(u,v)=(u,u+v)$.
The induced map $\bar h\: \bbr^2\to\bbr^2$
is is determined by the equation $\bar h\pcirc q=q\pcirc h$,
where $q\: \bbr^2\onto\bbr^2$ is the complex squaring map $q(u,v) = (u^2-v^2, 2uv)$.
Hence, 
$\bar h (x,0) = (0,2x)$, if $x \geq 0$, and $\bar h (x,0)= (x,0)$, if $x<0$. 
In particular, $\bar h$ is not differentiable at $x=0$. The non-compactness of $\bbr^2$ is not 
the point: one can transport this example onto the Riemann sphere. 
\end{Remark}

\begin{remark}
The smooth structure given by \lemref{smoothquo}  on $X/G$ is not the same as the
\emph{functional smooth structure}  on $X/G$ induced by the orbit map (see Bredon  
\cite[p.~301]{bredon1}), nor is it the same as the smooth stratifold
structure  induced by the $G$-action on $X$ (see Kreck \cite{kreck-book1}).
Both of these structures are functorial, unlike the structure given by \lemref{smoothquo}, but neither one gives $X/G$ the structure of a smooth manifold.
\end{remark}

In spite of \remref{R.nonfunct}, one has the following uniqueness result.

\begin{Lemma}\label{smoothquoUNI}
Let $X$ and $X'$ be two smooth closed  $G$-manifolds with codimension $2$ fixed point sets.
Suppose that $X$ and $X'$ are $G$-equivariantly diffeomorphic.
Then, the smooth structures on $X/G$ and $X'/G$ given by \textup{\lemref{smoothquo}} are
diffeomorphic.
\end{Lemma}

\begin{proof}
Let $h\:X\to X'$ be a $G$-equivariant diffeomorphism and
$\bar h\:X/G\to X'/G$ be the induced homeomorphism. Let $g$ be a $G$-invariant
Riemannian metric on $X$ and let $g'=h_*g$ be the metric on $X'$ transported by $h$. 
With these metrics, $h$ is an isometry
and the construction of Step~1 in the proof of \lemref{smoothquo} implies
that $\bar h\: (X/G)_g \to  (X'/G)_{g'}$ is a diffeomorphism. The result then follows
from Step~2 in the proof of \lemref{smoothquo}. 
\end{proof}

As an application of the same ideas, we give the following ``descent" result for a smooth action of a compact Lie group $H$ on $X$ which commutes with the $G$-action. Note that such an action induces a topological $H$-action on
$X/G$.
\begin{Lemma}\label{smoothquo-equiv}
Let $X$ be a smooth  $G$-manifold such that  the fixed point set $X^G$ is
a closed manifold of codimension $2$. Suppose that $X$ is equipped with
a smooth action of a compact Lie group $H$ which
commutes
with the $G$-action.
Then, there exists a smooth $H$-manifold $M$ and a
$H$-equivariant homeomorphism
$h\:X/G\to M$ such that the composed map $X\onto X/G \fl{h} M$ is a
branched covering
with branched locus the image of $X^G$. If $X$ is closed, the manifold $M$
is unique up to $H$-equivariant diffeomorphism.
\end{Lemma}
\begin{proof}
We use an $H$-invariant Riemannian metric $g$ on $X$ and apply \lemref{smoothquo} again.
\end{proof}
\begin{example}\label{classic2}
One of the classical conjugation $4$-manifolds is $S^2\times S^2$ with involution given by complex conjugation on each factor.  
Note that complex conjugation on $S^2$ may be expressed as the reflection $(x,y,z) \mapsto (x, y, -z)$, and this involution commutes with the standard $S^1$-action given by rotation in the $xy$-plane. Therefore the quotient $S^2\times S^2/G$ inherits an effective $T^2$-action, which is smooth with respect to the smooth structure provided by  \lemref{smoothquo-equiv}. However, Orlik \cite{orlik1} applied the classification of smooth $T^2$-actions by Orlik and Raymond \cite{orlik-raymond1} to show that a smooth homotopy $4$-sphere with an effective smooth $T^2$-action must be the standard $S^4$.
\end{example}

\begin{ccote}\textbf{Lifted structure}.
We now consider the opposite problem: to show that a smooth structure on the quotient of a $2$-fold branched covering induces a canonical smooth $G$-action on the total space.

\begin{Lemma}[Smooth $G$-action]\label{brcov2invo}
Let $p\: (X,Y) \to (M,N)$ be a branched $2$-fold covering with branched locus $N$,
where $M$ is a smooth closed manifold, and $N$ is a smooth closed submanifold of codimension 
$2$ in $M$.
Then $X$ admits a smooth $G$-action with $X^G=Y$ such that the smooth structure
on $X/G$ given by \textup{\lemref{smoothquo}} is diffeomorphic to $M$. 
\end{Lemma}

\begin{proof}
The involution on $X-Y$ is the deck transformation of the covering $X-Y\to M-N$. This is smooth with respect to the induced smooth structure on $X-Y$ from the covering.
Around $Y$, the map $p$ is modelled by~\eqref{E-bran}, and we obtain a smooth structure on $X$. The deck transformation
of $q\: P\times_{O(2)} D^\ddstar \to P\bar \times_{O(2)} D^\ddstar$ is given by 
$(a,z)\mapsto (a,-z)$. It extends to a smooth $G$-action on $X$ with $X^G=Y$.  

If $M$ is a closed manifold, so is $X$. One has a commutative diagram
$$
\begin{array}{c}{\xymatrix@C-3pt@M+2pt@R-4pt{%
 & X \ar@{>>}[dl]_(0.50){\pi}
\ar[dr]^(0.50){q}  
\\
X/G \ar[rr]^(0.50){h}_(0.50){\approx}  &&
M
}}\end{array} \ .
$$ 
Since $h$ is a continuous bijection between compact spaces, $h$ is a homeomorphism.
Hence, $M$ is a smooth structure on $X/G$ satisfying~(1) of \lemref{smoothquo}.
The result follows from part~(2) of the same lemma.
\end{proof}

A closed tubular neighbourhood of a codimension $2$ submanifold $N \subset M$ will be called a $D$-tube (since it is diffeomorphic to a smooth fibre bundle with fibre $D$). The next result is our version of Durfee and Kauffman \cite[Prop.~1.1]{durfee-kauffman1}.

\begin{Lemma}[Existence]\label{P.Exibran}
Let $M$ be a smooth closed manifold, and let $N$ be a smooth closed submanifold of codimension 
$2$ in $M$. Let $\dot p\:\dot X\to M-N$ be a smooth $2$-fold covering. Suppose that, for a $D$-tube around $N$, 
the preimage by $\dot p$ of each $D^\ddstar$-fiber is connected. 
Then
\begin{enumerate}
\item The covering $\dot p$ extends to a smooth $2$-fold branched covering $p\:X\to M$
with branched locus $N$. 

\item If $p\:X\to M$ and $p'\:X'\to M$ are two such branched coverings,
then $X'$ is $G$-diffeomorphic to $X$ (for the smooth structures and $G$-action defined in \textup{\lemref{brcov2invo}}). 
\end{enumerate}

\end{Lemma}

\begin{proof}
Choose a Riemannian metric $\bar g$ on $M$. 
This associates a smooth principal $O(2)$-bundle $Q$ to the normal bundle to $N$.
We may suppose that $\bar g$ is calibrated around $N$, meaning that the exponential
map defines a smooth embedding $\bar\varphi\: Q\times_{O(2)} D \to M$. 
Denote by $\dot p\: L\to Q\times_{O(2)} D^\ddstar$ the $2$-fold covering induced from $\dot p\:\dot X\to M-N$
by the embedding $\bar\varphi$. Consider the pull-back diagram
\eqncount
\begin{equation}\label{P.Exibran-eq10}
\begin{array}{c}{\xymatrix@C-3pt@M+2pt@R-4pt{%
\check L \ar[d]^(0.50){\check p}
\ar[r]  &
L \ar[d]^(0.50){\dot p}  \\
Q\times D^\ddstar \ar[r] &
Q\times_{O(2)} D^\ddstar
}}\end{array} .
\end{equation}
and denote by $\tilde D^\ddstar \to D^\ddstar$ the map $z\mapsto z^2$ from $D^\ddstar$ to itself.
Choosing a point $a\in Q$ gives a base point $(a,1)\in Q\times D^\ddstar$, and we let
let $\check D^\ddstar=\check p^\mun(\{a\}\times D^\ddstar)$. 

The rotation vector
field $\xi$ on $D^\ddstar$, defined by $\xi_z=(z,iz)\in D^\ddstar\times \bbc \approx TD^\ddstar$,
lifts to a smooth vector field $\check\xi$ on $\check D^\ddstar$, which is complete
since $\xi$ is. Consider the radius path given by the inclusion $\rho\:(0,1]\to D^\ddstar$.
Choose a lifting $\check\rho\:(0,1]\to \check D^\ddstar$ and  
integrate $\check\xi$ with these initial conditions. 
By our assumption on $D^\ddstar$-fibers, this will produce a $G$-diffeomorphism
$\tilde D^\ddstar\fl{\approx}\check D^\ddstar$ over the identity of $D^\ddstar$.

Over the slice $Q\times\{1\}$, the map 
$\check p$ is a $2$-fold covering $\tilde Q\to Q$. 
We deduce that there is a $G$-diffeomorphism 
$\check\beta\:\tilde Q \times_G \tilde D^\ddstar \fl{\approx} \check L$
over the identity of $Q\times D^\ddstar$.

Let $a\in Q$ and let $\tilde a\in \check p^\mun(a)$. 
Because of diagram~\eqref{P.Exibran-eq10} and  the relation 
$(a\alpha,z)\sim (a,\alpha z)$ in the definition of $Q\times_{O(2)} D^\ddstar$,
there is a commutative diagram
\eqncount
\begin{equation}\label{P.Exibran-eq40}
\begin{array}{c}{\xymatrix@C-3pt@M+2pt@R-4pt{%
O(2) \ar[d]^(0.50){\psi}
\ar[r]^(0.34){\approx}  &
\check p^\mun(a\cdot O(2)) \ar[d]^(0.50){\check p}  \\
O(2) \ar[r]^(0.50){\approx}  &
a\cdot O(2)
}}\end{array} 
\end{equation}
where $\psi$ is the epimorphism defined in \eqref{E-defpsi} (this is a $2$-fold covering).
We deduce that $\tilde Q$ is a smooth principal $O(2)$-bundle and that
$Q \approx \tilde Q\bar\times_{O(2)} O(2)$. 
Hence
$$
Q\times_{O(2)} D^\ddstar \approx  [\tilde Q\bar\times_{O(2)} O(2)]\times_{O(2)} D^\ddstar
\approx  \widetilde Q\bar\times_{O(2)} D^\ddstar \, .
$$
Suppose first that the normal bundle $\nu$  to $N$ is not orientable. 
We claim that there is a $G$-diffeomorphism 
$\beta\:\tilde Q \times_{O(2)} \tilde D^\ddstar\to L$ making the following diagram
commutative:
$$
\begin{array}{c}{\xymatrix@C-3pt@M+2pt@R-4pt{%
\tilde Q \times_G \tilde D^\ddstar \ar[dr]^(0.55){\check\beta}_(0.57){\approx}  
 \ar@<-1mm>[ddr]
\ar[rrr]  &&&
\tilde Q \times_{O(2)} \tilde D^\ddstar \ar@{.>}[dl]_(0.55){\beta}^(0.55){\approx}
\ar[dd]
\\
& \check L \ar[d]^(0.50){\check p}
\ar[r]  &
 L \ar[d]^(0.50){\dot p}  \\
& Q\times D^\ddstar \ar[r] &
Q\times_{O(2)} D^\ddstar   & 
\tilde Q\, \bar\times_{O(2)} D^\ddstar \ar[l]_(0.46){\approx}
}}\end{array} .
$$
Indeed, if $\nu$ is non-orientable, then $Q$ is connected. 
As $\tilde D$ is connected, we deduce from  diagram~\eqref{P.Exibran-eq40}
that $\check L$ is connected. Hence, any orbit for the diagonal $O(2)$-action
on $\tilde Q\times_G\tilde D^\ddstar$ goes to a single point in $L$. This guarantees
that $\check\beta$ descends to $\beta$. With these constructions,
the covering projection $\dot p\:\dot X\to M-N$ now extends to a smooth branched
covering $p\:X\to M$ where
\eqncount
\begin{equation}\label{P.Exibran-eq50}
 X= \dot X \cup_\beta \tilde Q \times_{O(2)} D  \, .
\end{equation}
This proves the existence of $p\:X\to M$ when $\nu$ is non-orientable. 
in the other case, we do the whole proof above, replacing $O(2)$ by 
the connected group $SO(2)$. 

For the uniqueness statement of \lemref{P.Exibran}, observe that the
smooth structure on $X$ given by the decomposition~\eqref{P.Exibran-eq50}
is associated to the Riemannian metric $\bar g$ on $M$. A proof of the
uniqueness statement of \lemref{P.Exibran} may thus be obtained in a process analogous to
Steps~2 and~3 of the proof of \lemref{smoothquo} (see also the uniqueness statement
in \cite[Prop.~1.1]{durfee-kauffman1} and its proof).
\end{proof}
\end{ccote}
\begin{remark}\label{R.Exibran}
Suppose that, in \lemref{P.Exibran}, $N$ is connected and let $V$ be a $D$-tube around $N$.
By the homotopy exact sequence of the bundle $D^\ddstar\to V-N \to N$,
the condition on the $D$-fibers is equivalent to $\dot p^\mun(V-N)$ being connected. 
If this is not the case, the proof of \lemref{P.Exibran} shows that
$\dot p$ extends to an unbranched $2$-fold covering $X\to M$. 
\end{remark}

As in \remref{R.nonfunct}, smooth branched coverings are not functorial,
see \cite[\S~1]{durfee-kauffman1}. However, as in \lemref{smoothquoUNI}, one has
the following uniqueness result.

\begin{Lemma}[Uniqueness]\label{L.Unibran}
Let $(M,N)$ and $(M',N')$ be two manifold pairs, where $M$ and $M'$ are closed
and $N,N'$ are closed submanifolds of codimension $2$. Suppose that 
there is a diffeomorphism $h\:(M,N)\to (M',N')$. Then, the smooth branched
coverings over $M$ and and $M'$, with branched locus $N$ and $N'$, are diffeomorphic. 
\end{Lemma}

\begin{proof}
Let $X\to M$ and $X'\to M'$ be two such smooth branched coverings.
Then the pull-back $h^*X'\to M$ is a smooth branched covering over $M$,
with branched locus $N$, obviously diffeomorphic to $X'$. By \lemref{P.Exibran},
$h^*X'$ is diffeomorphic to $X$.
\end{proof}

\providecommand{\bysame}{\leavevmode\hbox to3em{\hrulefill}\thinspace}
\providecommand{\MR}{\relax\ifhmode\unskip\space\fi MR }
\providecommand{\MRhref}[2]{%
  \href{http://www.ams.org/mathscinet-getitem?mr=#1}{#2}
}
\providecommand{\href}[2]{#2}

\end{document}